\documentclass{amsart}

\usepackage{amsthm}
\usepackage{amssymb}

\newtheorem{thma}{Theorem} 
\newtheorem{thmaa}{Theorem} 

\newtheorem{thm}{Theorem}[section]
\newtheorem{cor}[thm]{Corollary}
\newtheorem{lem}[thm]{Lemma}
\newtheorem{prop}[thm]{Proposition}

\theoremstyle{definition}
\newtheorem{defn}[thm]{Definition}

\theoremstyle{remark}
\newtheorem{rem}[thm]{Remark}
\newtheorem{con}[thm]{Convention}

\DeclareMathOperator{\Aut}{Aut}
\DeclareMathOperator{\Out}{Out}
\DeclareMathOperator{\Inn}{Inn}
\DeclareMathOperator{\diam}{diam}
\DeclareMathOperator{\CV}{CV}
\DeclareMathOperator{\cv}{cv}
\DeclareMathOperator{\st}{st}
\DeclareMathOperator{\val}{val}
\DeclareMathOperator{\PF}{PF}
\DeclareMathOperator{\rk}{rk}
\DeclareMathOperator{\BCC}{BCC}
\DeclareMathOperator{\Isom}{Isom}

\title{non-rigidity of cyclic automorphic orbits in free groups}
\author{Brian Ray}
\address{Department of Mathematics, University of Illinois at Urbana-Champaign, 1409 West Green Street, Urbana, IL 61801, USA}
\email{ray8@illinois.edu}
\thanks{The author acknowledges support from National Science Foundation grants 
DMS 0838434 ÓEMSW21MCTP: Research Experience for Graduate StudentsÓ, and DMS-0904200.}

\begin{document}

\begin{abstract}
We say a subset $\Sigma \subseteq F_N$ of the free group of rank $N$ is \emph{spectrally rigid} if whenever $T_1, T_2 \in \cv_N$ are $\mathbb{R}$-trees in (unprojectivized) outer space for which $\| \sigma \|_{T_1} = \| \sigma \|_{T_2}$ for every $\sigma \in \Sigma$, then $T_1 = T_2$ in $\cv_N$.  The general theory of (non-abelian) actions of groups on $\mathbb{R}$-trees establishes that $T \in \cv_N$ is uniquely determined by its translation length function $\| \cdot \|_T \colon F_N \to \mathbb{R}$, and consequently that $F_N$ itself is spectrally rigid.  Results of Smillie and Vogtmann \cite{MR1182503}, and of Cohen, Lustig, and Steiner \cite{MR1105334} establish that no finite $\Sigma$ is spectrally rigid.  Capitalizing on their constructions, we prove that for any $\Phi \in \Aut(F_N)$ and $g \in F_N$, the set $\Sigma = \{ \Phi^n(g) \}_{n \in \mathbb{Z}}$ is not spectrally rigid.
\end{abstract}

\maketitle

\section{Introduction}\label{Introduction}

The notion of a ``marked length spectrum'' initially arose in the context of Riemannian and metric geometry, specifically in the setting of a group $G$ acting by isometries on a metric space $X$ with some negative curvature properties.  We view such an action as a homomorphism $\rho \colon G \to \Isom(X)$.  To each nontrivial $g \in G$, we associate the minimal displacement of the isometry $\rho(g) \colon X \to X$, thus obtaining a function $\| \cdot \|_\rho \colon G \to \mathbb{R}_{\geq 0}$ defined by $\| g \|_\rho := \inf \{ d(x,gx), x \in X \}$.  This situation often appears in the context of Riemannian geometry.  If $M$ is a closed connected Riemannian manifold of (not necessarily constant) negative curvature, then the Riemannian metric $\rho$ determines a homomorphism $\rho \colon \pi_1(M) \to \Isom(\widetilde{M})$, where $\widetilde{M}$ is the universal cover of $M$.  In this situation, the minimal displacement function $\| \cdot \|_\rho \colon \pi_1(M) \to \mathbb{R}_{\geq 0}$ is called the \emph{marked length spectrum} of $\rho$.  The extent to which the marked length spectrum of $\rho$ determines the isometry type of $(M,\rho)$ is at the center of much research (see, for example, the introduction in \cite{2010arXiv1001.1729K}).

The classic \emph{marked length spectrum rigidity conjecture} asserts that knowing the marked length spectrum of $\rho$ --- where $\rho$ is a smooth Riemannian metric of negative curvature on a closed manifold, $M$ --- does in fact determine the isometry type of $(M,\rho)$.  In settings where the marked length spectrum rigidity conjecture holds, one can ask whether or not the isometry type of $(M,\rho)$ is determined by the restriction $\| \cdot \|_\rho \colon H \to \mathbb{R}_{\geq 0}$ where $H$ is a \emph{proper} subset of $\pi_1(M)$.  Naturally, we say such an $H$ is a ``spectrally rigid'' subset of $\pi_1(M)$.

In this paper, we consider the setting where $G = F_N$, a finitely generated free group, and $X$ is an $\mathbb{R}$-tree.  More specifically, we work in \emph{Culler-Vogtmann outer space}, denoted $\cv_N$, which is the space of all free minimal discrete isometric actions of $F_N$ on $\mathbb{R}$-trees, up to $F_N$-equivariant isometry.  Equivalently, such a $T \in \cv_N$ can be described by means of its quotient metric graph $G = T / F_N$ (with marking isomorphism $F_N \cong \pi_1(G)$), where we record the lengths of the edges in $G$ so that $T$ is identified with its universal cover $\widetilde{G}$.  For each $T \in \cv_N$, we have a \emph{translation length function} $\| \cdot \|_T \colon F_N \to \mathbb{R}_{\geq 0}$.  As with the marked length spectrum associated to Riemannian manifolds, the translation length function is a class function, and $\| g \|_T$ is equal to the translation length of the isometry $g \colon T \to T$.

Following Kapovich \cite{2010arXiv1001.1729K}, we define a subset $\Sigma \subseteq F_N$ to be \emph{spectrally rigid} if whenever $T_1, T_2 \in \cv_N$ are such that $\|\sigma\|_{T_1} = \|\sigma\|_{T_2}$ holds for every $\sigma \in \Sigma$, it follows that $T_1 = T_2$ in $\cv_N$.  Results of Smillie and Vogtmann \cite{MR1182503} and Cohen, Lustig, and Steiner \cite{MR1105334} show that no finite subset of $F_N$ is spectrally rigid; however, since the action of $F_N$ on $T$ is free it is non-abelian ($\| ghg^{-1} h^{-1} \|_T \neq 0$, for some $g, h \in F_N$), and so the translation length function uniquely determines the action (see \cite[\S3, Theorem 4.1]{MR1851337}).  Thus $F_N$ itself is spectrally rigid.  Motivated by these results, Kapovich \cite{2010arXiv1001.1729K} initiated the search for (non) spectrally rigid infinite but ``sparse'' subsets of $F_N$, proving in \cite{2010arXiv1001.1729K} for $N \geq 2$ that almost every trajectory of a simple non-backtracking random walk is spectrally rigid.  Furthermore, Carette, Francaviglia, Kapovich, and Martino \cite{2011arXiv1106.0688F}, show that for $N \geq 2$, the set of primitive elements in $F_N$ is spectrally rigid.  They also show in \cite{2011arXiv1106.0688F} that for any subgroup $H \leq \Aut(F_N)$ $(N \geq 3)$ which projects to an infinite normal subgroup in $\Out(F_N)$, the orbit $Hg$ of any nontrivial $g \in F_N$ is spectrally rigid.  Kapovich \cite{2010arXiv1001.1729K} remarks without proof that in the case of an atoroidal fully irreducible automorphism $\Phi \in \Aut(F_N)$ ($N \geq 3$), one can show --- using a previous result of Smillie and Vogtmann \cite{MR1182503} together with train track tecnhiques --- that for any nontrivial $g \in F_N$, the orbit $\langle \Phi \rangle g$ is not spectrally rigid.  These results naturally lead to the following question, posed in \cite{2011arXiv1106.0688F}: Is it true that for any $H \leq \Aut(F_N)$, either for all nontrivial $g \in F_N$ the orbit $Hg$ is spectrally rigid, or for all nontrivial $g \in F_N$ the orbit $Hg$ is not spectrally rigid.  Carette, Francaviglia, Kapovich, and Martino conjectured (\cite[Conjecture 7.5]{2011arXiv1106.0688F}) that in the case $H \leq \Aut(F_N)$ is cyclic, the orbit $Hg$ is never spectrally rigid.  Our main result, Theorem \ref{thma}, verifies this conjecture.

\begin{thma}\label{thma}
Let $N \geq 2$.  Let $\Phi \in \Aut(F_N)$.  Let $g \in F_N$ be arbitrary.  Then the set $\Sigma = \{ \Phi^n(g) \}_{n \in \mathbb{Z}}$ is not spectrally rigid.
\end{thma}

The proof of Theorem \ref{thma} requires two main ingredients: improved relative train track machinery established by Bestvina, Feighn, and Handel (see \cite{MR1765705}, \cite{MR1147956}); and previous results regarding outer space by Smillie and Vogtmann \cite{MR1182503}, and by Cohen, Lustig, and Steiner \cite{MR1105334}.  The latter results show that no finite subset of $F_N$ is spectrally rigid (\cite{MR1182503} handles $N \geq 3$, and \cite{MR1105334} $N=2$).  A critical observation of Kapovich \cite{2010arXiv1001.1729K} reveals that these arguments hold for more general subsets of $F_N$, not necessarily finite.  In particular, we use the argument in \cite{MR1182503} to show if $\Sigma \subseteq F_N$ ($N \geq 3$) satisfies Property $\mathcal{W}$ (see Definition \ref{W}), then $\Sigma$ is not spectrally rigid.  Similarly, we formulate Property $\mathcal{W}^*$ (see Definition \ref{Ws}) and use the arguments in \cite{MR1105334} to show that any subset $\Sigma \subseteq F_2$ which satisfies property $\mathcal{W}^*$ is not spectrally rigid.  We further show that Property $\mathcal{W}^*$ follows from Property $\mathcal{W}$ (see Proposition \ref{WW}).  As mentioned above, obtaining the result for iwip automorphisms is straightforward.  It is the extension of these ideas to non-iwip automorphisms which occupies the majority of this paper.

The main outline of our argument is as follows.  After setting the framework in Section \ref{Preliminaries}, we define (for a given $\Phi \in \Aut(F_N)$) a property (Property $\mathcal{P}$, see Definition \ref{P}) so that if $\Phi \in \Aut(F_N)$ has Property $\mathcal{P}$, then for any $g \in F_N$ the set $\Sigma = \{ \Phi^n(g) \}_{n \in \mathbb{Z}}$ will have Property $\mathcal{W}$, and therefore be non-spectrally rigid.  Verification of Property $\mathcal{P}$ for a given $\Phi \in \Aut(F_N)$ will require detailed analysis of a (relative) train track $f \colon G \to G$ representing $[\Phi] \in \Out(F_N)$.  In Section \ref{EG-Automorphisms}, we handle those $[\Phi]$ that can be equipped with relative train tracks with exponentially growing top stratum; Section \ref{NEG-Automorphisms} handles those with non-exponentially growing top stratum.  In Section \ref{Fully Irreducible Automorphisms} we cover the iwip case in order to elucidate the arguments which are central to the main result.  In particular, we show that given an iwip $\Phi \in \Aut(F_N)$ with train track $f \colon G \to G$, $g \in F_N$, and a primitive $a \in F_N$, there is a uniform bound $M$ so that if $(\tau(a))^k$ occurs as a subword in the reduced form of $f^n(\tau(g))$ (for any $n \geq 1$), then $|k| \leq M$.  Here $\tau \colon R_N \to G$ is the marking.  In Section \ref{EG-Automorphisms} and Section \ref{NEG-Automorphisms}, we show that appropriately modified versions of the above statement hold.  Handling the case $N=2$ involves showing that Property $\mathcal{W}$ suffices for Property $\mathcal{W}^*$, and is deferred to Section \ref{The Case of F2}.

\subsubsection*{Acknowledgements} The author is indebted to his advisor, Ilya Kapovich, for direction, counsel, and financial support.  The author is grateful to Patrick Reynolds and Brent Solie for several invaluable discussions.

\section{Preliminaries}\label{Preliminaries}

\subsection{Graphs, Fundamental Groups, and Automorphisms}\label{Graphs, Fundamental Groups, and Automorphisms}

As for the general objects of geometric group theory, we will  follow the expositions in \cite{MR2396717} and \cite{MR2365352}.  By a \emph{graph}, we mean a five-tuple $G = (V(G), E(G), o,t,-)$, where $V(G)$ is a nonempty set of vertices, $E(G)$ is a set of edges, and $o, t \colon E(G) \to V(G)$, $- \colon E(G) \to E(G)$ are functions whose image of an edge is its origin, terminus, and inverse, respectively.  Furthermore, we require $\overline{\overline{e}} = e, \overline{e} \neq e$, and $o(e) = t(\overline{e})$ for every $e \in E(G)$.  The \emph{$N$-rose}, denoted $R_N$, is the graph with $V(R_N) = \{ \ast \}$ and $\#(E(R_N)) = 2N$.  A \emph{subgraph} $H$ of a graph $G$ is a graph for which $V(H) \subseteq V(G)$, $E(H) \subseteq E(G)$, and the maps $o,t,-$ are restrictions to $E(H)$.  The \emph{star} of a vertex, denoted $\st_G v$ is defined as $\st_G v := \{ e \in E(G) : o(e) = v \}$; the \emph{valency} of $v$, denoted $\val_G v$, is the cardinality of its star.  An \emph{orientation} of $G$ is a decomposition $E(G) = E(G)^+ \sqcup E(G)^-$ where for each pair $\{e,\overline{e}\}$ of mutually inverse edges, exactly one edge belongs to $E(G)^+$.

By an \emph{path}, we mean either a vertex $v \in V(G)$ (in which case $p$ is a \emph{degenerate} path), or a sequence $p = e_0e_1 \cdots e_{n-1}$ of edges with $n \geq 1$ for which $t(e_i) = o(e_{i+1})$ for $0 \leq i \leq n-2$.  The \emph{inverse} of $p$, denoted $\overline{p}$, is defined to be either the vertex $v$ or the path $\overline{p} = \overline{e_{n-1}} \cdots \overline{e_1} \ \overline{e_0}$, depending on whether or not $p$ is degenerate.  By a \emph{subpath}, $p'$ of the path $p = e_0 e_1 \cdots e_{n-1}$, we mean an path $p' = e_i e_{i+1} \cdots e_{i+m}$, where $i \in \{0,1, \dots, n-1\}$ and $m \leq n-1-i$; we write $p' \Subset p$.  The \emph{length} of the path $p = e_0e_1 \cdots e_{n-1}$, denoted $|p|$, is $n$; a degenerate path has length zero.  The set of all paths in $G$ is denoted by $P(G)$.  For $p \in P(G)$ we define $o(p) := o(e_0)$, $t(p) := t(e_{n-1})$.  A \emph{closed} path is one for which $o(p) = t(p)$; in this case, we say $o(p) = t(p)$ is the \emph{basepoint} of the path $p$.  By a \emph{cyclic path}, we mean the equivalence class (under cyclic permutation) of a given closed path.  When dealing with cyclic paths, we modify the definition of subpath accordingly.  For $v \in V(G)$, let $P(G,v) := \{ p \in P(G)  :  o(p) = t(p) = v \}$.  An path is \emph{reduced} if it is \emph{degenerate} or $\overline{e_i} \neq e_{i+1}$ for $0 \leq i \leq n-2$.  A closed path is \emph{cyclically reduced} if it is reduced, and $\overline{e_0} \neq e_{n-1}$.  The process by which one removes subpaths of the form $\overline{e} e$ is called \emph{(cyclic) reduction}.  A \emph{forest} is a graph without nontrivial cyclically reduced paths; a connected forest is a \emph{tree}.  In case $T \subseteq G$ is a tree with $v, v' \in V(T)$, we denote by $[v,v']_T$ the unique path $p \subseteq T$ with $o(p) = v$ and $t(p) = v'$.

We say two paths $p_1, p_2$ are \emph{homotopic (relative endpoints)} if there is an path $p$ which can be obtained from both $p_1$ and $p_2$ by reduction.  Consequently, two such paths satisfy $o(p_1) = o(p_2)$, $t(p_1) = t(p_2)$.  Similarly, we say two cyclic paths $p_1, p_2$ are \emph{freely homotopic} if there is a cyclic path $p$ which can be obtained from both $p_1$ and $p_2$ by cyclic reduction.  For a path $p$, we denote by $[p]_h$ the \emph{homotopy class} of the path $p$; that is, the set of all paths homotopic to $p$ (relative endpoints).  By $[p]$ we mean the reduced form of $p$.  In the case that $p$ is a cyclic path, we denote its \emph{free homotopy class} by $[[p]]_h$, and by $[[p]]$ the equivalence class containing the cyclically reduced form of $p$.  We extend our notation and terminology to elements in $F_N$ as follows.  Given a (not necessarily reduced) word $w \in F_N$ expressed over some free basis $\mathcal{A}$ of $F_N$, consider now the path $w$ in the $N$-rose, $R_N$, where each edge $e \in E(R_N)^+$ is labelled by a basis element $a \in \mathcal{A}$.  Additionally, given a free basis $\mathcal{A}$ of $F_N$, and a (not necessarily reduced) word $w$, we write $[w]_\mathcal{A}$ for the reduced form of $w$ where $w$ is expressed as a word over $\mathcal{A}$, and we write $[[w]]_{\mathcal{A}}$ for the equivalence class containing the cyclically reduced form of $w$ over $\mathcal{A}$.

The \emph{fundamental group} of $G$ with respect to the vertex $v$, $\pi_1(G,v) := \{ [p]_h : p \in P(G,v) \}$.  By a \emph{graph map}, we mean a function $f \colon V(G_1) \cup E(G_1) \to V(G_2) \cup P(G_2)$ which sends $V(G_1)$ to $V(G_2)$, $E(G_1)$ to $P(G_2)$, and for which
\[
f(o(e)) = o(f(e)) \quad \text{and} \quad f(\overline{e}) = \overline{f(e)}
\]
for every $e \in E(G_1)$.  We write $f \colon G_1 \to G_2$; if $v_1 \in V(G_1)$ with $f(v_1) = v_2 \in V(G_2)$, we sometimes write $f \colon (G_1,v_1) \to (G_2,v_2)$ for emphasis. We extend graph maps to paths $p \in P(G_1)$ as follows:
\[
f(p) = f(e_0 e_1 \cdots e_{n-1}) := f(e_0) f(e_1) \cdots f(e_{n-1})
\]
If for each $e \in E(G_1)$, we have that $f(e)$ is a nondegenerate reduced path, then we say that the map $f$ is \emph{tight}.  For a graph map $f \colon (G_1,v_1) \to (G_2,v_2)$ with $G_1, G_2$ connected, there is an induced homomorphism $f_\# \colon \pi_1(G_1,v_1) \to \pi_1(G_2,v_2)$ given by the rule $[p]_h \mapsto [f(p)]_h$ for $p \in P(G_1,v_1)$.  In the case that $f_\#$ is an isomorphism, we say $f$ is a \emph{homotopy equivalence}, and that the isomorphism $f_\#$ is \emph{realized} by the homotopy equivalence $f$.  

Given a graph map $f \colon (G,v) \to (G,f(v))$, there is a family of induced homomorphisms $f^u_\# \colon \pi_1(G,v) \to \pi_1(G,v)$ given by $[p]_h \mapsto [u f(p) \overline{u}]_h$ where $u$ is any path with $o(u) = v$, $t(u) = f(v)$.  If $f^u_\# \in \Aut(\pi_1(G_1,v))$ for some $u$, then $f^u_\# \in \Aut(\pi_1(G_1,v))$ for all $u$.  Furthermore, $f^u_\#$ and $f^{u'}_\#$ differ only by conjugation and thus represent the same outer automorphism of $\pi_1(G_1,v)$; we denote this outer automorphism by $f_{\circledast}$.  

By a \emph{marked graph} we mean a pair $(G,\tau)$, where $G$ is a graph and $\tau \colon R_N \to G$ is a homotopy equivalence.  Choose a homotopy equivalence $\sigma \colon G \to R_N$ so that $(\tau \circ \sigma)_\# \in \Inn(\pi_1(G,v))$ (for some choice of $v$) and $(\sigma \circ \tau)_\# \in \Inn(\pi_1(R_N)) = \Inn(F_N)$.  To a given homotopy equivalence $f \colon G \to G$, we associate $(\sigma \circ f \circ \tau)_\circledast \in \Out(F_N)$, the outer automorphism \emph{determined} by $f$.  Given $\varphi \in \Out(F_N)$, we say that a homotopy equivalence $f \colon G \to G$ is a \emph{topological representative} of $\varphi$ with respect to $\tau$ if $f$ determines $\varphi$ and $f$ is tight.  For example, take $\Phi \in \Aut(F_N)$, and let $F_N = F(x_0,x_1, \dots, x_{N-1}), E(R_N)^+ = \{ e_0,e_1, \dots, e_{N-1} \}$.  Write $\Phi(x_i) = w_i(x_0,x_1, \dots, x_{N-1})$.  Let $\tau , \sigma \colon R_N \to R_N$ be the identity graph maps, and $f \colon R_N \to R_N$ be $f(e_i) = p_i(e_0,e_1, \dots, e_{N-1})$ where $p_i$ is obtained from $w_i$ by replacing $x_j$ by $e_j$.  Evidently $f$ is a topological representative of $\varphi$ with respect to the identity marking.

\subsection{Train Tracks}\label{Train Tracks}

We follow \cite{MR2396717}, \cite{MR1147956}, and \cite{MR1765705} for our exposition regarding relative train tracks.  A \emph{turn} in a connected graph $G$ is an unordered pair (say $e, e'$) of (not necessarily distinct) edges in $E(G)$, for which $o(e) = o(e')$.  A turn is \emph{degenerate} if $e=e'$ and \emph{non-degenerate} otherwise.  The set of all turns in $G$ is denoted by $T(G)$.  Given a tight graph map $f \colon G \to G$, we define a function $Df \colon E(G) \to E(G)$ by $Df(e) := e^*$ where $e^*$ is the initial edge in the path $f(e)$.   Define a function $Tf \colon T(G) \to T(G)$ by $Tf(e,e') := (Df(e),Df(e'))$.  A \emph{legal} turn is one for which $(Tf)^n(e,e')$ are non-degenerate for all $n \geq 0$; a turn is \emph{illegal} if it is not legal.  We say that an path $p = e_0 e_1 \cdots e_{n-1}$ is \emph{legal} if the turns $\overline{e_i},e_{i+1}$ are legal for $0 \leq i \leq n-2$.  These notions culminate in the following definition.

\begin{defn}[Train track map]\label{tt}
A graph map $f \colon G \to G$ is called a \emph{train track map} if for every $e \in E(G)$, $f(e)$ is nondegenerate and legal.
\end{defn}

\begin{rem}\label{ttr}
For every $e \in E(G)$, a train track map satisfies $[f^n(e)] = f^n(e)$, for all $n \geq 1$.
\end{rem}

Let $f \colon G \to G$ be a graph map.  Choose an orientation for $G$, and write $E(G)^+ = \{ e_0, e_1, \dots, e_{k-1} \}$.  We define a $k \times k$ matrix called the \emph{transition matrix} of $f$, denoted $M(f)$, as follows: $(M(f))_{ij}$ equals the number of occurrences of $e_i$ or $\overline{e_i}$ in $f(e_j)$.  We say that $\varphi \in \Out(F_N)$ is \emph{reducible} if there is a free factorization $F_N = F^0 \ast \cdots \ast F^{l-1} \ast H$, with $l \geq1$ and $1 \leq \rk (F^0) < N$, so that $\varphi$ permutes the conjugacy classes of the $F^i$; we allow $H = 1$.  Otherwise, we say $\varphi$ is \emph{irreducible}.  There is an alternate definition given in terms of irreducibility of transition matrices (see \cite{MR2396717} or \cite{MR1147956}).    We say that $\varphi \in \Out(F_N)$ is \emph{fully irreducible} (or \emph{irreducible with irreducible powers}, \emph{iwip} for short) if $\varphi^n$ is irreducible for every $n \geq 1$.  Thus $\varphi \in \Out(F_N)$ is not an iwip if and only if there is $n \geq 1$ so that $\varphi^n$ leaves invariant the conjugacy class of a proper free factor of $F_N$.  A result of Bestvina and Handel \cite{MR1147956} states that every irreducible $\varphi \in \Out(F_N)$ is topologically represented by a train track map.  Furthermore, every automorphism can be topologically represented by a improved relative train track map, a notion we recall below.

Let $f \colon G \to G$ be a topological representative for $\varphi \in \Out(F_N)$.  A \emph{filtration} for $f$ is a series of $f$-invariant subgraphs: $\varnothing = G_0 \subsetneq G_1 \subsetneq \cdots \subsetneq G_m = G$.  The \emph{$i^{th}$-stratum}, denoted $H_i$, is defined by $H_i := G_i \backslash G_{i-1}$, whereby (for a graph $G$ and a subgraph $H$) we mean the following:
\[
E(G \backslash H) := E(G) \backslash E(H)
\]
\[
V(G \backslash H) := \{ v \in V(G) : \exists e \in E(G \backslash H) \ \text{for which} \ o(e) = v \}
\]
By abuse of notation, if $S \subset E(G)$, we define graphs $G \backslash S$ and $G \cup S$ in the obvious manner.  A turn for which one edge is in $H_i$ and the other in $G_{i-1}$ is called \emph{mixed}.  A path $p \subseteq G_i$ is called \emph{$i$-legal} if its only illegal turns lie in $G_{i-1}$.  By only considering edges from $E(H_i)$ we obtain the \emph{transition submatrix} for $H_i$, denoted $M_i(f)$.  If $M_i(f)$ is the zero matrix, we say that $H_i$ is a \emph{zero} stratum.  Recall (\cite[Appendix]{MR2396717}) that a non-negative irreducible integer matrix has an associated Perron-Frobenius eigenvalue $\lambda \geq 1$.  If $\PF(M_i(f)) > 1$ (resp. $=1$) we say that $H_i$ is \emph{exponentially-growing} (resp. \emph{non-exponentially-growing}).  We follow \cite{MR1765705} for our definition of a relative train track.

\begin{defn}[Relative train track]\label{rtt}
A topological representative $f \colon G \to G$ of $\varphi \in \Out(F_N)$ is a \emph{relative train track} with respect to the filtration $\varnothing = G_0 \subsetneq G_1 \subsetneq \cdots \subsetneq G_m = G$ if the following criteria hold:
\begin{enumerate}
\item $G$ has no valence one vertices.
\item Each non-zero $M_i(f)$ is irreducible.
\item Each exponentially growing $H_i$ satisfies:
\begin{enumerate}
\item If $e \in E(H_i)$, then $Df(e) \in E(H_i)$.
\item If $p$ is a nontrivial path in $G_{i-1}$ for which $o(p) \in V(H_i)$ and $t(p) \in V(H_i)$, then $[f(p)]$ is nontrivial.
\item If $p$ is a legal path in $H_i$, then $f(p)$ is an $i$-legal path in $G_i$.
\end{enumerate}
\end{enumerate}
\end{defn}

It follows immediately from Definition \ref{rtt} that mixed turns are legal.  Furthermore, we note the following property of relative train tracks.

\begin{lem}[{Bestvina, Feighn, and Handel \cite[Lemma 2.5.2]{MR1765705}}]\label{BFH2.5.2}
Suppose that $f \colon G \to G$ is a relative train track map, that $H_r$ is an exponentially growing stratum and that $\sigma = a_1b_1a_2 \cdots b_l$ is a decomposition of an $r$-legal path into subpaths where each $a_i \subseteq H_r$ and each $b_j \subseteq G_{r-1}$.  (Allow the possibility that $a_1$ or $b_l$ is trivial, but assume that the other subpaths are nontrivial.)  Then
\[
[f(\sigma)] = f(a_1) [f(b_1)] f(a_2) \cdots [f(b_l)]
\]
and $[f(\sigma)]$ is $r$-legal.
\end{lem}

\qed

Furthermore, if we allow ourselves to work with an iterate of  $\varphi$, we can arrange for additional properties, thus obtaining an \emph{improved relative train track} (see \cite{MR1765705}).  As we will not be invoking the full power of the improved relative train track, we only state the properties of interest.

\begin{thm}[{Bestvina, Feighn, and Handel \cite[Theorem 5.1.5]{MR1765705}}]\label{BFH5.1.5}
For every outer automorphism $\varphi$ there is a relative train track map $f \colon G \to G$ representing a positive iterate of $\varphi$ with the following properties.
\begin{enumerate}
\item For every exponentially growing $H_i$, we have $M_i(f)$ is aperiodic; that is, there exists $k \geq 1$ for which all entries in $M_i^k(f)$ are positive.
\item If $H_i$ is a zero stratum, then $H_i$ is not a top stratum (i.e. $G_i \neq G)$.
\item If $H_i$ is a non-exponentially-growing stratum, then
\begin{enumerate}
\item $E(H_i) = e_0$, a single edge.
\item $f(e_0) = e_0 u$ where $u$ is a closed path in $G_{i-1}$ whose basepoint $t(e_0) = o(u) = t(u)$ is fixed by $f$.
\end{enumerate}
\end{enumerate}
\end{thm}

\begin{defn}[Good relative train track]\label{grtt}
A relative train track satisfying the conclusion of Theorem \ref{BFH5.1.5} will be called a \emph{good relative train track}.  
\end{defn}

\begin{rem}\label{iterate}
If $f \colon G \to G$ is a good relative train track representing $\varphi^n$ for $n \geq 1$, then $f^k \colon G \to G$ is a good relative train track representing $\varphi^{nk}$.
\end{rem}

By an \emph{EG-automorphism} we mean an automorphism $\varphi$ for which there exists $n \geq 1$ so that $\varphi^n$ can be equipped with a good relative train track with exponentially-growing top stratum.  We similarly define a \emph{NEG-automorphism}.  Note that iwip automorphisms are EG-automorphisms, but the converse does not hold.

\subsection{Culler-Vogtmann outer space}\label{Culler-Vogtmann outer space}

There are multiple equivalent descriptions of Culler-Vogtmann outer space.  Our subsequent definition of spectral rigidity treats outer space as a space of actions on trees.  The expositions in \cite{MR1182503} and \cite{MR1105334}, however, lend themselves nicely to the description of outer space in terms of marked metric graphs.  We will briefly describe both notions, explaining how one can move easily between the two.  We will follow \cite{MR1950871} and \cite{2010arXiv1001.1729K}.

\emph{Culler-Vogtmann outer space}, denoted by $\cv_N$, is the space of free minimal discrete isometric actions of $F_N$ on $\mathbb{R}$-trees, up to equivalence, where $T_1 \sim T_2$ if there exists an $F_N$-equivariant isometry $f \colon T_1 \to T_2$.  By \emph{projectivized Culler-Vogtmann outer space}, we mean the subset $\CV_N \subseteq \cv_N$ consisting of those trees $T$ for which the quotient metric graph $T / F_N$ has volume 1.  For each $g \in F_N$ and $T \in \cv_N$, we define \emph{translation length}, denoted $\|g\|_T$, by $\|g\|_T := \inf \{ d(x,gx) \colon x \in T \}$.  This gives a \emph{translation length function} $\| \cdot \|_T \colon F_N \to \mathbb{R}_{\geq 0}$.  We remark that $T \in \cv_N$ is uniquely determined (up to $F_N$-equivariant isometry) by $\| \cdot \|_T \colon F_N \to \mathbb{R}_{\geq 0}$ (see \cite[\S3, Theorem 4.1]{MR1851337}).

Alternatively, a point in $\cv_N$ is the equivalence class of a triple $(G,\tau,l)$ where
\begin{enumerate}
\item $G$ is a marked graph with respect to the marking $\tau \colon R_N \to G$.
\item For every $v \in V(G)$, we have $\val_G v \geq 3$.
\item Each $e \in E(G)$ is assigned a positive real number, its \emph{length}, denoted $l(e)$.  We require $l(e) = l(\overline{e})$.
\end{enumerate}

We refer to such a triple as a \emph{marked metric graph structure}.  Condition 3 allows us to treat $G$ as a metric space via the path metric.  The equivalence relation is as follows: $(G,\tau,l) \sim (G',\tau',l')$ if there is an isometry $h \colon G \to G'$ for which the following holds: there exists an path $u$ with $o(u) = (h \circ \tau)(\ast)$ and $t(u) = \tau'(\ast)$ for which $(h \circ \tau)_\# = c_u \circ \tau'_\#$.  Here $c_u \colon \pi_1(G',\tau'(\ast)) \to \pi_1(G',(h \circ \tau)(\ast))$ is $[p]_h \mapsto [u p \overline{u}]_h$.

Given such a $(G,\tau,l)$, the marking gives an isomorphism $F_N \cong \pi_1(G)$ and thus induces an action of $F_N$ on $\widetilde{G}$ by deck transformation.  This action is free, minimal, discrete, and by isometries, the metric on $\widetilde{G}$ being that lifted from $G$.  Furthermore, $\| g \|_T$ is the length of the shortest loop in the free homotopy class determined by $g$ in the quotient metric graph $T / F_N$.

\subsection{Bounded Cancellation}\label{Bounded Cancellation}

Suppose $f \colon G \to G$ and $f' \colon G' \to G'$ are relative train tracks for $\varphi$ and $\varphi^{-1}$, respectively.  In Section \ref{EG-Automorphisms}, we construct a primitive $a \in F_N$, so that the realizations $[[\tau(a)]]$ and $[[\tau'(a)]]$ in $G$ and $G'$, respectively, satisfy certain properties.  Verification of these properties relies on our ability to control cancellation in a manner we make precise in Proposition \ref{graph1}.  A result of Cooper \cite{MR916179} ensures this is possible.  We state a version best fit for our purposes.

\begin{lem}[{Bestvina, Feighn, and Handel \cite[Lemma 2.3.1]{MR1765705}}]\label{BFH2.3.1}
For any homotopy equivalence $f \colon G \to G'$ of marked graphs there is a constant $\BCC(f) \geq 0$ for which the following holds.  If $p = \alpha \beta$ is a path in $G$, then $[f(p)]$ is obtained from $[f(\alpha)]$ and $[f(\beta)]$ by concatenating and by cancelling $c \leq \BCC(f)$ edges from the terminal end of $[f(\alpha)]$ with $c$ edges from the initial end of $[f(\beta)]$.
\end{lem}  \qed

\subsection{Spectral Rigidity}\label{Spectral Rigidity}

We follow the exposition in \cite{2010arXiv1001.1729K}.  

\begin{defn}[Spectrally rigid]\label{Spectrally rigid}
We say $\Sigma \subseteq F_N$ is \emph{spectrally rigid} if whenever $T_1, T_2 \in \cv_N$ are such that $\|g\|_{T_1} = \|g\|_{T_2}$ for every $g \in \Sigma$, then $T_1 = T_2$ in $\cv_N$.
\end{defn}

In \cite{MR1182503}, Smillie and Vogtmann proved for $N \geq 3$ that no finite subset of $F_N$ is spectrally rigid.  It was remarked by Kapovich \cite{2010arXiv1001.1729K} that their arguments apply to a more general $\Sigma \subseteq F_N$, not necessarily finite.  In particular, if $\Sigma \subseteq F_N$ satisfies a``weak aperiodicity'' property, then $\Sigma$ is not spectrally rigid.  We will refer to this as property $\mathcal{W}$ (see Definition \ref{W} and Proposition \ref{SV}).  The argument in \cite{MR1182503} involved constructing a particular free basis $\mathcal{A}$ of $F_N$, and a particular marked metric graph $G$ with $\pi_1(G) \cong F_N$, so that the realization (in $G$) of each $\sigma \in \Sigma$ satisfy certain properties.  These properties allow us to perturb the length data of $G$ in a way that does not change the translation lengths of elements in $\Sigma$, yet does produce different trees in outer space.  We remark that these marked metric graphs have volume 1, and so represent points in $\CV_N$.

\begin{defn}[Property $\mathcal{W}$]\label{W}
Let $\Sigma \subseteq F_N$.  We say $\Sigma$ has \emph{property} $\mathcal{W}$ if there exist a free basis $\mathcal{A}$ of $F_N$, $a \in \mathcal{A}$, and $M \geq 1$ so that for any $\sigma \in \Sigma$, if $a^k \Subset [[\sigma]]_\mathcal{A}$ then $|k| \leq M$.
\end{defn}

In view of the argument in \cite{MR1182503}, the following is immediate.

\begin{prop}\label{SV}
Suppose $N \geq 3$, and $\Sigma \subseteq F_N$ has property $\mathcal{W}$.  Then $\Sigma$ is not spectrally rigid.
\end{prop} \qed

Cohen, Lustig, and Steiner \cite{MR1105334} showed that no finite subset of $F_2$ is spectrally rigid.  Broadly speaking, their argument was similar to that in \cite{MR1182503}; however, it was necessary to perturb the volumes of the quotient graphs.  Definition \ref{Spectrally rigid} allows for such a modification since $(G,\tau, l) \neq (G,\tau,\lambda l)$ in $\cv_N$ for $\lambda \neq 1$.  Our investigation of the argument in \cite{MR1105334} reveals that property $\mathcal{W}$ does indeed suffice for non spectral rigidity in the case $N = 2$.  To make this connection clear, we introduce an intermediate property (property $\mathcal{W}^*$ (see Definition \ref{Ws}).  We defer the details to Section \ref{The Case of F2}.

\section{Property $\mathcal{P}$}\label{}

We begin with a definition which is central to the entire paper.  For convenience, we write $\mathbb{N} = \{0,1,2,\dots\}$.

\begin{defn}[Property $\mathcal{P}$]\label{P}
Given $\Phi \in \Aut(F_N)$, we say $\Phi$ has \emph{property} $\mathcal{P}^{+}(\mathcal{A},a)$ if there exist a free basis $\mathcal{A}$ of $F_N$ and $a \in \mathcal{A}$, so that for any $g \in F_N$, there exists $M = M_g \geq 1$ so that if $a^k \Subset h \in \{ [[\Phi^n(g)]]_\mathcal{A} \}_{n \in \mathbb{N}}$ then $|k| \leq M$.  We say $\Phi$ has property $\mathcal{P}^{-}(\mathcal{A},a)$ if the above holds with $\mathbb{N}$ replaced by $-\mathbb{N}$, and we say $\Phi$ has property $\mathcal{P}(\mathcal{A},a)$ if the above holds with $\mathbb{N}$ replaced by $\mathbb{Z}$.  We say $\Phi$ has $\mathcal{P}$, (respectively $\mathcal{P}^+, \mathcal{P}^-$) if $\Phi$ has $\mathcal{P}(\mathcal{A},a)$, (respectively $\mathcal{P}^+(\mathcal{A},a), \mathcal{P}^-(\mathcal{A},a)$) for some pair $(\mathcal{A},a)$.  We say $\varphi \in \Out(F_N)$ has $\mathcal{P}(\mathcal{A},a)$ (respectively $\mathcal{P}^+(\mathcal{A},a), \mathcal{P}^-(\mathcal{A},a)$) if for some (equivalently, for all) $\Psi \in \varphi$, we have that $\Psi$ has $\mathcal{P}(\mathcal{A},a)$ (respectively $\mathcal{P}^+(\mathcal{A},a), \mathcal{P}^-(\mathcal{A},a)$).  Finally, we say $\varphi$ has $\mathcal{P}$, (respectively $\mathcal{P}^+, \mathcal{P}^-$) if $\varphi$ has $\mathcal{P}(\mathcal{A},a)$, (respectively $\mathcal{P}^+(\mathcal{A},a), \mathcal{P}^-(\mathcal{A},a)$) for some pair $(\mathcal{A},a)$.
\end{defn}

\begin{rem}\label{pos-neg}
If $\Phi$ has both $\mathcal{P}^+(\mathcal{A},a)$, and $\mathcal{P}^-(\mathcal{A},a)$, then $\Phi$ has $\mathcal{P}(\mathcal{A},a)$.
\end{rem}

Definition \ref{P} and Definition \ref{W} immediately yield the following proposition.

\begin{prop}\label{PW}
If $\varphi \in \Out(F_N)$ has $\mathcal{P}$, then for any $\Phi \in \varphi$ and for any $g \in F_N$ the set $\Sigma = \{ \Phi^n(g) \}_{n \in \mathbb{Z}}$ has $\mathcal{W}$.
\end{prop} \qed

As indicated in Theorem \ref{BFH5.1.5}, good relative train tracks often represent a positive iterate of a given outer automorphism.  The following lemma and its corollaries ensure that this does not pose a problem.

\begin{lem}\label{ptp}
Suppose $\Psi \in \Aut(F_N)$ has $\mathcal{P}$ and $\Phi^l = \Psi$ for some $l \geq 1$.  Then $\Phi$ has $\mathcal{P}$.
\end{lem}

\begin{proof}
For any $m \in \mathbb{Z}$, write $m = ql + r$ with $0 \leq r < l$.  Then for $g \in F_N$, we have
\[
\Phi^m (g) = \Phi^{ql + r} (g) = \Phi^{ql} ( \Phi^r (g) ) = \Psi^q ( \Phi^r (g) )
\]
Let $\Gamma = \{ \Phi^r(g) \}_{r=0}^{l-1}$.  Say $\Psi$ has $\mathcal{P} = \mathcal{P}(\mathcal{A},a)$.  For each $\Phi^r(g) \in \Gamma$, we obtain $M_r$ for which
\[
\text{if} \quad a^k \Subset h \in \{ [[ \Psi^n(\Phi^r(g)) ]]_{ \mathcal{A}} \}_{n \in \mathbb{Z}} \quad \text{then} \quad |k| \leq M_r
\]
Set $M' := \max_{0 \leq r < l-1} \{ M_r \}$.  Thus,
\[
\text{if} \quad a^k \Subset h \in \{ [[ \Phi^n(g) ]]_{\mathcal{A}} \}_{n \in \mathbb{Z}} \quad \text{then} \quad |k| \leq M'
\]
\end{proof}

Lemma \ref{ptp} immediately implies the following.

\begin{cor}\label{ptp-one-sided}
Suppose $\Psi \in \Aut(F_N)$ has $\mathcal{P}^{+}$ (respectively $\mathcal{P}^{-}$) and $\Phi^l = \Psi$ for some $l \geq 1$.  Then $\Phi$ has $\mathcal{P}^{+}$ (respectively $\mathcal{P}^{-}$).
\end{cor}  \qed

Lemma \ref{ptp} and Definition \ref{P} imply the following two corollaries.

\begin{cor}\label{ptp-out}
Suppose $\psi \in \Out(F_N)$ has $\mathcal{P}$ and $\varphi^l = \psi$ for some $l \geq 1$.  Then $\varphi$ has $\mathcal{P}$.
\end{cor} \qed

\begin{cor}\label{ptp-one-sided-out}
Suppose $\psi \in \Out(F_N)$ has $\mathcal{P}^{+}$ (respectively $\mathcal{P}^{-}$) and $\varphi^l = \psi$ for some $l \geq 1$.  Then $\varphi$ has $\mathcal{P}^{+}$ (respectively $\mathcal{P}^{-}$).
\end{cor} \qed

As $1 \in \Aut(F_N)$ has $\mathcal{P}$, we also obtain the following.

\begin{cor}\label{fin-ord}
If $\varphi \in \Out (F_N)$ is such that $\varphi^l = 1 \in \Out(F_N)$, then $\varphi$ has $\mathcal{P}$.
\end{cor}  \qed

We'll prove Theorem \ref{thma} by showing that for every $\Phi \in \Aut(F_N)$, there is some $l \geq 1$ so that $\Phi^l$ has $\mathcal{P}$ and then apply Lemma \ref{ptp}, Proposition \ref{PW} and Proposition \ref{SV}.  Incidently, Corollary \ref{fin-ord} establishes Theorem \ref{thma} for those $\Phi \in \Aut(F_N)$ for which $[\Phi]^l = 1 \in \Out(F_N)$ for some $l \geq 1$.  In view of Theorem \ref{BFH5.1.5}, Remark \ref{iterate}, and Lemma \ref{ptp} we establish the following convention.

\begin{con}\label{edge-image}
For an EG-automorphism with good relative train track $f$, we will assume that the $f$ image of any edge from $H_t$ contains all edges in $H_t$.  This is achieved by passing to a power as per Condition 1 in Theorem \ref{BFH5.1.5}.  Furthermore, we will speak of a good relative train track ``representing $\varphi$'' when we actually mean a good relative train track representing an \emph{iterate} of $\varphi$.
\end{con}

The following proposition enables us to use topological representatives and their associated marked graphs in order to verify property $\mathcal{P}$ for a given automorphism.  It makes use of properties of Gromov-hyperbolic spaces; in particular, the $F_N$-equivariant quasi-isometry that exists between universal covers of marked graphs (see, for example, \cite{MR1086649}).

\begin{prop}\label{qi}
Let $\tau \colon R_N \to G$ and $\tau' \colon R_N \to G'$ be marked graphs.  Let $a \in F_N$, and $\Sigma \subseteq F_N$.  Then the following are equivalent.

\begin{enumerate}
\item $\sup \{ |k| : \tau(a)^k \Subset \tau(\sigma), \sigma \in \Sigma \} < \infty$
\item $\sup \{ |k| : \tau'(a)^k \Subset \tau'(\sigma), \sigma \in \Sigma \} < \infty$
\end{enumerate}

\end{prop}  \qed

\begin{cor}\label{marked-P}
Let $f \colon G \to G$ be a topological representative for $\varphi \in \Out(F_N)$.  Suppose there exist a free basis $\mathcal{A}$ of $F_N$ and $a \in \mathcal{A}$ such that for any $g \in F_N$, there exists $M \geq 1$ so that if $[[\tau(a)]]^k \Subset h \in \{ [f^n(\tau(g))] \}_{n \in \mathbb{N} (\text{resp } - \mathbb{N} )}$ then $|k| \leq M$.  Then $\varphi$ has $\mathcal{P}^{+}(\mathcal{A},a)$ (respectively $\mathcal{P}^{-}(\mathcal{A},a)$).
\end{cor}

\begin{proof}
Let $\tau \colon R_N \to G$ be the marking for $G$, and $\tau' \colon R_N \to R_N$ be a homotopy equivalence realizing the basis $\mathcal{A}$ and invoke Proposition \ref{qi}.
\end{proof}

The following Lemma will be useful in our analysis of forward images of top stratum edges for EG-automorphisms.

\begin{lem}\label{eg-L}
Let $N \geq 2$.  Let $\varphi \in \Out(F_N)$ be an EG-automorphism.  Then there is $L \geq 1$ so that for sufficiently large $n$ and for all $e_i, e_j \in E(H_t)$, we have the following: if $w \Subset [f^n(e_i)]$ is such that $w$ contains more than $L$ edges (counted with repetition) from $E(H_t)$, then $e_j \Subset w$.
\end{lem}

\begin{proof}
For any $w \Subset [f^n(e)]$, let $|w|_t$ denote the number of top stratum edges (counted with repetition) in $w$.  Let $K := \max_{e \in E(H_T)} \{ |[f(e)]|_t \}$.  Let $L := 2K$.  Then for $w \Subset [f^n(e_i)]$ with $|w|_t \geq L$, we have that $[f^{-1}(w)] \Subset [f^{n-1}(e_i)]$, and $[f^{-1}(w)]$ contains an edge from $E(H_t)$.  Thus $w$ contains all edges from $E(H_t)$.

\end{proof}

\section{Fully Irreducible Automorphisms}\label{Fully Irreducible Automorphisms}

Although the iwip case follows from the result for EG-automorphisms, obtaining the result for EG-automorphisms is more involved, requiring several arguments that are not necessary to handle the iwip case.  In order to focus on the more essential points, we provide a standalone exposition for iwip automorphisms.

\begin{lem}\label{iwip-edge}
Let $N \geq 2$.  Let $\varphi \in \Out(F_N)$ be fully irreducible.  Let $g \in F_N$ be arbitrary.  Let $a \in F_N$ be any primitive element.  Let $\gamma := [[\tau(g)]]$, and $\alpha := [[\tau(a)]]$.  Let $f \colon G \to G$ be a train track map for $\varphi$.  Let $e_0 \in E(G)$ be arbitrary.  Suppose for all $l \geq 1$, there exists $m = m_l \geq 1$ for which $\alpha^l \Subset [f^{m_l}(\gamma)]$.  Then for all $k \geq 1$, there exists $n = n_k$ for which $\alpha^k \Subset f^{n_k}(e_0)$.  
\end{lem}

\begin{proof}
We may assume that as $j \to \infty$, we have $|[f^j(\gamma)]| \to \infty$.  Otherwise there exists a uniform bound on $|[f^j(\gamma)]|$ and so the set $\{ [f^j (\gamma)] \}_{j \in \mathbb{N}}$ is finite and the Lemma follows vacuously.  Write $\gamma = e_0e_1 \cdots e_{h-1}$.  We have
\[
[f^j(\gamma)] =  [f^j(e_0) f^j(e_1)\cdots f^j(e_{h-1}) ]
\]
by the train track property.  Note that reduction can only occur between the subpaths $f^j(e_i)$ and $f^j(e_{i+1 \bmod h})$.  For each $j$, write
\[
[f^j(\gamma)] = p^j_{i_1} p^j_{i_2} \cdots p^j_{i_l}
\]
where $0 \leq i_1 < i_2 < \cdots < i_r \leq h-1$, and $p^j_{i_q} \Subset f^j(e_q)$.  We may assume $j$ is large enough so that the set $I = \{ i_1, i_2, \dots, i_r \}$ is fixed, as this is eventually true.  For $q \in I$, let $M_q$ be so that $|p^j_{i_q}| \leq M_q$ for all $j$; if no such bound exists, set $M_i := 0$.  Note since $|[f^j(\gamma)]| \to \infty$, there is a $q$ for which $|p^j_{i_q}| \to \infty$.  Set $M := \sum_{q} M_q$.

Note that if $w \Subset [f^j(\gamma)]$ is such that $|w| > M$, then $w$ meets some $p^j_{i_q}$ for which $|p^j_{i_q}| \to \infty$.  Furthermore, if we let $C := \# \{ q \in I : |p^j_{i_q}| \to \infty \}$, then any $w$ with $|w| = M' > M$ meets at least $(M'-M)/C$ of a $p^j_{i_q}$ for which $|p^j_{i_q}| \to \infty$.  As $M, C$ are fixed, given $k$ we choose $l$ so that
\[
\frac{|\alpha^l| - M}{C} \geq |\alpha|^k
\]
Then $\alpha^k \Subset p^j_{i_q} \Subset f^j(e_{i_q})$ for some $q$.  Furthermore for a fixed $e_0 \in E(G)$ and any $q \in I$, if $a^k \Subset f^{n_k}(e_{i_q})$, then $a^k \Subset f^{n_k+1}(e_0)$.

\end{proof}

\begin{lem}\label{iwip-edge-bound}
Let $N \geq 2$.  Let $\varphi \in \Out(F_N)$ be fully irreducible.  Let $a \in F_N$ be any primitive element.  Let $\alpha := [[\tau(a)]]$.  Let $f \colon G \to G$ be a train track map for $\varphi$.  Let $e_0 \in E(G)$ be arbitrary.  Then there exists $k \geq 1$, such that for all $n \geq 1$, we have $\alpha^k \not \Subset f^n(e_0)$.
\end{lem}

\begin{proof}
Suppose for contradiction that for all $k \geq 1$, there exists $n = n_k \geq 1$ such that $a^k \Subset f^{n_k}(e_0)$.  Choose $k_1 \geq |\alpha| + 1$.  Let $L$ be as per Lemma \ref{eg-L}.  Then if $w \Subset f^j(e_0)$ is so that $|w| \geq L$, then $w$ contains all edges in $E(G)$.  Let $\lambda_- := \min_{e \in E(G)} \{ |f(e)| \}$, and $\lambda_+ := \max_{e \in E(G)} \{ |f(e)| \}$.  Choose $k_2$ so that 
\[
\lambda_+^{-(n_{k_1}+1)} |\alpha^{k_2}| \geq L
\]
Then $\alpha^{k_2} \Subset f^{n_{k_2}}(e_0)$ is so that
\[
e_0 \Subset f^{-(n_{k_1}+1)}(\alpha^{k_2}) \Subset f^{n_{k_2} -(n_{k_1}+1)}(e_0)
\]
As $\alpha^{k_1} \Subset f^{n_{k_1}}(e_0)$, we have
\[
f(\alpha^{k_1}) \Subset \alpha^{k_2} \Subset f^{n_{k_2}}(e_0)
\]
Since we are working in the forward image of an edge, the train track property gives
\[
f(\alpha^{k_1}) = f(\alpha)^{k_1}
\]
By our choice of $k_1$, the Pigeonhole Principle gives $m_1, m_2 \geq 1$ for which
\[
f(\alpha)^{m_1} = (\sim \alpha)^{m_2}
\]
where $\sim \alpha$ denotes a cyclic permutation of $\alpha$.  As $\alpha$ represents a primitive $a \in F_N$, we conclude that $m_2 = \pm m_1$.  We now note that $f(\alpha)^{m_1} = (\sim \alpha)^{\pm m_1}$ is impossible since $f$ expands path lengths by at least $\lambda_- \geq \#E(G) \geq 2$.

\end{proof}

We immediately obtain the following.

\begin{cor}\label{false}
Let $N \geq 2$.  Let $\varphi \in \Out(F_N)$ be fully irreducible.  Then for any free basis $\mathcal{A}$ of $F_N$, and for any $a \in \mathcal{A}$, $\varphi$ has $\mathcal{P}^{+}(\mathcal{A},a)$.
\end{cor}

\begin{proof}
Let $f \colon G \to G$ be a train track map for $\varphi$.  By Lemma \ref{iwip-edge-bound}, for any free basis $\mathcal{A}$ of $F_N$, and any $a \in \mathcal{A}$, there is $M \geq 1$ so that if $[[\tau(a)]]^k \Subset f^n(e_0)$ for any $n \geq 1$, then $|k| \leq M$.  Let $g \in F_N$ be arbitrary.  Then by Lemma \ref{iwip-edge}, there is $M' \geq 1$ so that if $[[\tau(a)]]^k \Subset [f^n([[\tau(g)]])]$ for any $n \geq 1$, then $|k| \leq M'$.  Thus by Corollary \ref{marked-P} and Corollary \ref{ptp-one-sided-out}, $\varphi$ has $\mathcal{P}^{+}(\mathcal{A},a)$.

\end{proof}

We can now prove that iwip automorphisms have property $\mathcal{P}$.

\begin{thm}\label{iwip-main}
Let $N \geq 2$.  Let $\varphi \in \Out(F_N)$ be fully irreducible.  Then $\varphi$ has $\mathcal{P}$.
\end{thm}

\begin{proof}
Note that $\varphi$ is an iwip if and only if $\varphi^{-1}$ is an iwip.  Thus by Corollary \ref{false} for any free basis $\mathcal{A}$, and for any $a \in \mathcal{A}$, both $\varphi$ and $\varphi^{-1}$ have $\mathcal{P}^{+}(\mathcal{A},a)$; that is, $\varphi$ has both $\mathcal{P}^{+}(\mathcal{A},a)$ and $\mathcal{P}^{-}(\mathcal{A},a)$, and so has $\mathcal{P}(\mathcal{A},a)$.  
\end{proof}

\section{EG-Automorphisms}\label{EG-Automorphisms}

The arguments in Lemma \ref{iwip-edge} also apply, \emph{mutadis mutandis}, to the case where $\varphi$ is an EG-automorphism.  We therefore omit the proof of the following Lemma.

\begin{lem}\label{red-edge}
Let $N \geq 2$.  Let $\varphi \in \Out(F_N)$ be an EG-automorphism and $f \colon G \to G$ a good relative train track.  Let $g \in F_N$ be arbitrary.  Let $a \in F_N$ be any primitive element for which $\alpha := [[\tau(a)]]$ crosses the exponentially growing top stratum.  Let $\gamma := [[\tau(g)]]$.  Let $e_0 \in E(H_t)$ be arbitrary.  Suppose $\forall l \geq 1, \exists m = m_l \geq 1$ for which $\alpha^l \Subset [f^{m_l}(\gamma)]$.  Then $\forall k \geq 1, \exists n = n_k$ for which $\alpha^k \Subset [f^{n_k}(e_0)]$.  
\end{lem}

\qed

Lemma \ref{red-edge-bound} is the analogue of Lemma \ref{iwip-edge-bound}.

\begin{lem}\label{red-edge-bound}
Let $N \geq 2$.  Let $\varphi \in \Out(F_N)$ be an EG-automorphism and $f \colon G \to G$ a good relative train track.  Let $a \in F_N$ be any primitive element for which $\alpha := [[\tau(a)]]$ crosses the exponentially growing top stratum.  Let $e_0 \in E(H_t)$ be arbitrary.  Then there exists $k \geq 1$, such that for all $n \geq 1$, we have $\alpha^k \not \Subset [f^n(e_0)]$.
\end{lem}

\begin{proof}
As in Lemma \ref{iwip-edge-bound} , we suppose for contradiction that for all $k \geq 1$, there exists $n = n_k \geq 1$ for which $\alpha^k \Subset [f^{n_k}(e_0)]$.  Choose $k_1 \geq |\alpha| + 1$.  Let $L$ be as per Lemma \ref{eg-L}.  Thus if $w \Subset [f^j(e_0)]$ is so that $|w|_t \geq L$, then $w$ contains all edges in $E(H_T)$.  Choose $k_2$ so that $|[f^{-(n_{k_1}+1)}(\alpha^{k_2})]|_t \geq L$.  Then $\alpha^{k_2} \Subset [f^{n_{k_2}}(e_0)]$ is so that
\[
e_0 \Subset [f^{-(n_{k_1}+1)}(\alpha^{k_2})] \Subset [f^{n_{k_2} -(n_{k_1}+1)}(e_0)]
\]
As $\alpha^{k_1} \Subset [f^{n_{k_1}}(e_0)]$, we have
\[
[f(\alpha^{k_1})] \Subset \alpha^{k_2} \Subset [f^{n_{k_2}}(e_0)]
\]
Write $[f(\alpha)] = u v u^{-1}$, so that $v = [f(\alpha)]_{c \#}$ as cyclic paths.  Then
\[
[f(\alpha^{k_1})] = (uvu^{-1})^{k_1} = u v^{k_1} u^{-1}
\]
By our choice of $k_1$, the Pigeonhole Principle gives $m_1,m_2 \geq 1$ for which
\[
(\sim[f(\alpha)])^{m_1} = v^{m_1} = (\sim \alpha)^{m_2}
\]
where $\sim$ denotes a cyclic permutation.  As $\alpha$ represents a primitive $a \in F_N$, we conclude that $m_2 = \pm m_1$.  We now note that $(\sim[f(\alpha)])^{m_1} = (\sim \alpha)^{\pm m_1}$ is impossible since $f$ expands path lengths in $H_t$ by at least $\lambda_- :=  \min_{e \in E(H_t)} \{ |f(e)|_t \} \geq 2$.

\end{proof}

The proof of Theorem \ref{iwip-main} does directly not carry over to non-iwip automorphisms.  A priori there is no guarantee that the realization (under the marking) of a given primitive $a \in F_N$ in good relative train tracks for $\varphi$ and $\varphi^{-1}$ must cross top strata in both.  However, we can arrange this for a particular choice of $a$, as is outlined below.

\begin{prop}\label{elt-prim}
Let $\mathcal{A} = \{a_0,a_1, \dots, a_{n-1}\}$ be a free basis of $F_N$.  Let $w \in F_N$ be so that there is exactly one occurrence of $a_{n-1}$ (or $\overline{a_{n-1}})$ in $[w]_\mathcal{A}$.  Then $w$ is a primitive element. 
\end{prop}

\begin{proof}
Note that $\mathcal{A}' = \{a_0, a_1, \dots, a_{n-2}\}$ is a free basis of $F_{N-1}$, and write $[w]_\mathcal{A} = \alpha a_{n-1} \beta$ with $\alpha = [\alpha]_{\mathcal{A}'}$, $\beta = [\beta]_{\mathcal{A}'}$.  Then the free basis $\mathcal{A}' \cup \{w\}$ is obtained from the free basis $\mathcal{A} = \mathcal{A}' \cup \{a_{n-1}\}$ by means of right (respectively left) transvections as determined by $\beta$ (respectively $\alpha$).  As transvections lie in $\Aut(F_N)$, the proposition follows.
\end{proof}

\begin{cor}\label{edge-prim}
Let $G$ be a finite graph.  Let $e_0 \in E(G)$.  If $\alpha$ is a cyclicly reduced path $G$ that crosses $e_0$ (in either direction) exactly once, then $\alpha$ is a primitive element in $\pi_1 (G) \cong F_N$.
\end{cor}

\begin{proof}
Note first that $e_0$ can not be a separating edge.  Choose a maximal tree $T \subset G \backslash \{e_0\}$.  Then $T$ is also a maximal tree for $G$ with $(G \backslash \{e_0\}) / T \cong R_{N-1}$ and $G/T \cong R_N$ as graphs.  Furthermore $(G \backslash \{e_0\})/ T \subset G/T$, the only difference being the additional loop associated to the edge $e_0$.  The cyclic path $[[\tau(\alpha)]]$ in $G/T$ crosses the loop $e_0$ exactly once (in either direction); the remainder of the edge path lies in $(G \backslash \{e_0\})/ T$.  
\end{proof}

The proof Lemma \ref{red-prim} will require us to choose loops in $G$ with very specific properties.  The following two lemmas ensure this is possible.

\begin{prop}\label{graph}
Let $N \geq 2$.  Let $G$ be a finite connected graph with $\pi_1(G) \cong F_N$ for which $G$ has no valence one vertices.  Let $H \subsetneq G$ be a proper subgraph with $\#E(H) \geq 1$.  Let $M \geq 1$, $e_0 \in E(H)$ be given.  Then there is a cyclically reduced path $\alpha$ such that $\alpha$ crosses $e_0$ at least $M$ times (in either direction), and $\alpha$ represents a primitive element in $\pi_1(G)$.
\end{prop}

\begin{proof}
Suppose first that $e_0 \in E(H)$ is a $G$-non-separating edge.  Choose a maximal tree $T \subset G \backslash \{e_0\}$.  Then $T$ is a maximal tree for $G$ with $\#E(G) - \#E(T) = N \geq 2$, and $e_0 \not \in E(T)$.  Choose $e \neq e_0 \in E(G) \backslash E(T)$.  Choose $v \in V(G)$.  By Corollary \ref{edge-prim} the cyclically reduced form of the following path suffices
\[
\{ [v,o(e_0)]_T \ e_0 \ [t(e_0), v]_T \}^M \ [v,o(e)]_T \ e \ [t(e),v]_T
\]
since all $M$ occurrences of $e_0$ remain after cyclic reduction, along with the occurrence of $e$.  Now suppose that $e_0 \in E(H)$ is a $G$-separating edge.  Let $\Gamma, \Gamma'$ be the two components of $G \backslash \{e_0\}$.  There are two cases to consider.  Note first that neither of $\Gamma, \Gamma'$ are trees, else $G$ would have valence one vertices. If either $\pi_1(\Gamma) \cong F_M$ or $\pi_1(\Gamma') \cong F_M$ with $M \geq 2$, then we proceed as follows.  Without loss suppose $\pi_1(\Gamma) \cong F_M$ with $M \geq 2$.  Choose maximal trees $T, T'$ in $\Gamma, \Gamma'$ respectively.  Note that $T^* := T \cup T' \cup \{e_0\}$ is a maximal tree for $G$.  Choose $e_1 \in E(\Gamma) \backslash E(T)$ and $e_2 \neq e_1 \in E(\Gamma) \backslash E(T)$.  Choose $e' \in E(\Gamma') \backslash E(T')$.  Choose $v \in V(\Gamma)$.  Then by Corollary \ref{edge-prim}, the cyclically reduced form of the following path suffices
\[
\{ [v,o(e_1)]_T \ e_1 \ [t(e_1) , o(e_0)]_T \ e_0 \ [t(e_0),o(e')]_{T'} \ e' \  [t(e'),t(e_0)]_{T'}
\]
\[
\overline{e_0} \ [o(e_0),v]_T \}^{ \lceil \frac{M}{2} \rceil } \ [v,o(e_2)]_T \ e_2 \ [t(e_2),v]_T
\]
since the $M$ occurrences of $e_0$, and $\overline{e_0}$, along with the single occurrence of $e_2$ remain after cyclic reduction.  Now suppose that $\pi_1(\Gamma) \cong \pi_1 (\Gamma') \cong \mathbb{Z}$.  Choose maximal trees $T, T'$ in $\Gamma, \Gamma'$ and let $T^* := T \cup T' \cup \{e_0\}$ be a maximal tree for $G$.  Choose $e_1 \in E(\Gamma) \backslash E(T)$ and $e' \in E(\Gamma') \backslash E(T')$.  Choose $v \in V(\Gamma)$.  As the map $G \to G/T^*$ is a homotopy equivalence, write $\pi_1(G) \cong F_2 = F(l_1, l')$, where
\[
l_1 := [v,o(e_1)]_T \ e_1 \ [t(e_1),v]_T
\]
\[
l' := [v,o(e_0)]_T \ e_0 \ [t(e_0),o(e')]_{T'} \ e' \ [t(e'),t(e_0)]_{T'} \ \overline{e_0} \ [o(e_0),v]_T
\]
Let $\eta_1, \eta' \colon F(l_1, l') \to F(l_1, l')$ be the Nielsen Transformations
\[
\eta_1 := \begin{cases}
l_1 \mapsto l_1 l' \\
l' \mapsto l' \end{cases}
\]
\[
\eta' := \begin{cases}
l_1 \mapsto l_1 \\
l' \mapsto l'l_1 \end{cases}
\]
In the automorphic image of $l_1$ under $(\eta' \circ \eta_1)^{M'}$ the number of occurrences of $l_1 l'$ and $l' l_1$ is unbounded in $M'$.  Each such occurrence crosses $e_0$, as the $e_0$ do not cancel during reduction.  Thus choosing $M'$ large enough, the closed path $(\eta' \circ \eta_1)^{M'}(l_1)$ suffices.

\end{proof}

\begin{prop}\label{graph1}
Let $N \geq 2$.  Let $G$ be a finite connected graph with $\pi_1(G) \cong F_N$ for which $G$ has no valence one vertices.  Let $H \subsetneq G$ be a proper subgraph with $\#E(H) \geq 1$.  Let $\alpha$ be a cyclically reduced path in $G \backslash H$.  Then there is a closed path $\alpha ' $ for which the following hold: 

\begin{enumerate}
\item $\alpha' = \eta \ \alpha$ as a cyclic path for some closed path $\eta$.
\item $[[ \alpha']]$ crosses $H$; in particular $[\eta]$ crosses $H$ and all occurrences of H-edges in $[\eta]$ remain after cyclic reduction to $[[\alpha']]$.
\item $[[\alpha']]$ represents a primitive element in $\pi_1(G)$.
\item There is a bound depending only on $G$ for the amount of cyclic reduction in cyclically reducing $[\eta] \alpha$ to $[[\alpha']]$.
\end{enumerate}
\end{prop}

\begin{proof}
Let $\Gamma_1$ be the component of $G \backslash H$ for which $\alpha \subset \Gamma_1$.  Let $\Delta = G \backslash \{ \Gamma_1 \cup H \}$, and write $\Delta = \cup_{i=2}^k \Gamma_i$, where the $\Gamma_i$ are the components of $\Delta$.  Let $\Delta^* \subset \Delta$ consist of those components which satisfy the following: there is an edge $e \in E(H)$ for which $\Gamma_i \cup \{e\}$ is connected and meets $\Gamma_1$.   If there is $e_0 \in E(H)$ for which $o(e_0) \in E(\Gamma_1)$ and $t(e_0) \in E(\Gamma_1)$, then we proceed as follows.  Let $T_1$ be a maximal tree in $\Gamma_1$.  Choose a vertex $v \Subset \alpha$ and write $o(\alpha) = t(\alpha) = v$.  Then the cyclically reduced form of the following path suffices
\[
\alpha' = \underbrace{[v,o(e_0)]_{T_1} \ e_0 \ [t(e_0),v]}_{\eta} \ \alpha
\]
Note the amount of cyclic reduction between $[\eta]$ and $\alpha$ is at most $2 \diam \Gamma_1$, and $e_0 \Subset [[\alpha']]$ occurs exactly once (see Corollary \ref{edge-prim}).  Suppose now that no such $e_0 \in E(H)$ exists; that is if $e \in E(H)$ is such that $o(e) \in V(\Gamma_1)$, then $t(e) \in V(\Delta^*)$.  Let $e_0 \in E(H)$ be such that $o(e_0) \in V(\Gamma_1)$ and write $t(e_0) \in V(\Gamma_j)$ for $j \geq 2$.  Suppose further that $\Gamma_j$ is not a tree.  Let $T_j$ be a maximal tree in $\Gamma_j$, and let $e_j \in E(\Gamma_j) \backslash E(T_j)$.  Then the cyclically reduced form of the following path suffices
\[
\alpha' = \underbrace{[v,o(e_0)]_{T_1} \ e_0 \ [t(e_0),o(e_j)]_{T_j} \ e_j \ [t(e_j),t(e_0)]_{T_j} \ \overline{e_0} \ [o(e_0),v]_{T_1}}_{\eta} \ \alpha
\]
Note the amount of cyclic reduction between $[\eta]$ and $\alpha$ is at most $2 \diam \Gamma_1$.  Also, $[[\alpha']]$ crosses $e_j$ exactly once (see Corollary \ref{edge-prim}).  Finally suppose that for every $e \in E(H)$ for which $o(e) \in V(\Gamma_1)$, we have that $t(e)$ lies in $V(\Delta^*)$, and $\Delta^*$ is a forest.  Since $G$ is connected, either there is an edge $e^* \in E(H)$ for which $o(e^*) \in V(\Gamma_i)$, $t(e^*) \in V(\Gamma_j)$ with $i \neq j$ and both $\Gamma_i, \Gamma_j$ are trees in $\Delta^*$ or there is a pair of edges $e_0, e_1$ with $o(e_0), o(e_1) \in V(\Gamma_1)$ and $t(e_0), t(e_1) \in V(\Gamma_i)$ with $\Gamma_i$ a tree in $\Delta^*$.  In the former case, the cyclically reduced form of the following path suffices
\[
\alpha' = \underbrace{[v,o(e_i)]_{T_1} \ e_i \ [t(e_i),o(e^*)]_{\Gamma_i} \ e^* \ [t(e^*),t(e_j)]_{\Gamma_j} \ \overline{e_j} \ [o(e_j),v]_{T_1}}_{\eta} \ \beta
\]
In the latter case, the cyclically reduced form of the following path suffices
\[
\alpha' = \underbrace{[v,o(e_0)]_{T_1} \ e_0 \ [t(e_0),t(e_1)]_{T_i} \ \overline{e_2} \ [o(e_2),v]_{T_1}}_{\eta} \ \beta
\]
\end{proof}

\begin{lem}\label{red-prim}
Let $N \geq 2$.  Let $\varphi \in \Out(F_N)$ be such that neither $\varphi$ nor $\varphi^{-1}$ is an NEG-automorphism.  Then there is a free basis $\mathcal{A}$ and an $a \in \mathcal{A}$ for which $\alpha := [[\tau(a)]]$ crosses $H_t$ and $\alpha' := [[\tau'(a)]]$ crosses $H'_{t'}$.
\end{lem}

\begin{proof}
Let $f \colon G \to G$, $f' \colon G' \to G'$ be as per hypothesis.  By Proposition \ref{graph}, there exists a cyclically reduced path $\alpha$ in $G$ such that $\alpha$ crosses $e_0 \in E(H_t)$ at least $M$ times ($M$ to be chosen later), and represents a primitive element in $F_N$.  Write $\alpha = [[\tau(a)]]$ for $a \in F_N$.  If $\alpha' := [[\tau'(a)]]$ crosses $H'_{t'}$, then we are done.  Else assume that $\alpha' \subseteq G'_{t'-1}$.  Let $v \Subset \alpha'$ be a vertex, and write $o(\alpha') = t(\alpha') = v$.  Define $\alpha'' := \eta \ \alpha'$ where $\eta$ is as per Proposition \ref{graph1}.  Let $\upsilon \colon G' \to G$ be the difference of markings.  Note $[[\upsilon(\alpha')]] = \alpha$.  Let $\BCC(\upsilon)$ be the bounded cancellation constant for $\upsilon$ (see Lemma \ref{BFH2.3.1}).  Choose $M > 2(\diam G' + \BCC(\upsilon))$ so as to ensure enough of $[\upsilon(\alpha')]$ remains when we cyclically reduce to $[[\upsilon(\alpha'')]]$.  More precisely, write
\[
[[\upsilon(\alpha'')]] = [[\upsilon(\eta \ \alpha')]] = [[ \ [\upsilon(\eta)] \ [\upsilon(\alpha')] \ ]]
\]
By Lemma \ref{graph1}, cyclic reduction between $[\eta]$ and $\alpha$ is at most $2 \diam G'$.  By Lemma \ref{BFH2.3.1}, cyclic reduction between $[\upsilon(\eta)]$ and $[\upsilon(\alpha')]$ is at most $2 \BCC(\upsilon)$.  Since the number of occurrences of $e_0$ in $\alpha = [[\upsilon(\alpha')]]$ is greater than $M$, at least one occurrence of $e_0$ must remain in $[[\upsilon(\alpha'')]]$.

\end{proof}

\begin{rem}\label{rem-red-prim}
The conclusion of Lemma \ref{red-prim} holds for weaker hypotheses, however we only make use of it in its current form.
\end{rem}

\begin{thm}\label{red-eg-main}
Let $N \geq 2$.  Let $\varphi \in \Out(F_N)$ be such that neither $\varphi$ nor $\varphi^{-1}$ is an NEG-automorphism.  Then $\varphi$ has $\mathcal{P}$.
\end{thm}

\begin{proof}
Let $\mathcal{A}$, and $a \in \mathcal{A}$ be as in the conclusion of Lemma \ref{red-prim}.  Let $f \colon G \to G, f' \colon G' \to G'$ be good relative train tracks for $\varphi, \varphi^{-1}$, respectively.  Let $g \in F_N$ be arbitrary.  By Lemma \ref{red-edge-bound} and Lemma \ref{red-edge}, there is $M \geq 1$ so that if $[[\tau(a)]]^k \Subset [f^n([[\tau(g)]])]$ for any $n \geq 1$, then $|k| \leq M$.  Thus by Corollary \ref{marked-P} and Corollary \ref{ptp-one-sided-out}, $\varphi$ has $\mathcal{P}^{+}(\mathcal{A},a)$.  Similarly, Lemma \ref{red-edge-bound}, Lemma \ref{red-edge}, Corollary \ref{marked-P}, and Corollary \ref{ptp-one-sided-out} imply that $\varphi^{-1}$ has $\mathcal{P}^{+}(\mathcal{A},a)$, that is, $\varphi$ has $\mathcal{P}^{-}(\mathcal{A},a)$.  Thus $\varphi$ has $\mathcal{P}(\mathcal{A},a)$.  
\end{proof}

\section{NEG-Automorphisms}\label{NEG-Automorphisms}

\begin{lem}\label{red-edge-neg}
Let $N \geq 2$.  Let $\varphi \in \Out(F_N)$ be an NEG-automorphism and $f \colon G \to G$ a good relative train track for $\varphi$.  Let $g \in F_N$ be arbitrary.  Let $\gamma := [[\tau(g)]]$.  Let $a \in F_N$ be any primitive element for which $\alpha := [[\tau(a)]]$ crosses the top strata, $H_t$.  Then $\exists k \geq 1$, such that $\forall n \geq 1$, we have $\alpha^k \not \Subset [f^n(\gamma)]$.
\end{lem}

\begin{proof}
By Theorem \ref{BFH5.1.5}, we may assume that $H_t$ consists of a single edge, $e_0$, and that $f(e_0) = e_0 u$ where $u \subseteq G_{t-1}$ is a closed path whose basepoint is fixed by $f$.  Thus 
\begin{align*}
[f^j(e_0)] = [f^{j-1}(e_0 u)] &= [f^{j-1}(e_0)] [f^{j-1}(u)] \\
&= [f^{j-2}(e_0 u)] [f^{j-1}(u)]  \\
&= [f^{j-2}(e_0)] [f^{j-2}(u)] [f^{j-1}(u)] \\
&\vdots \\
&= e_0 u [f(u)] [f^2(u)] \cdots [f^{j-1}(u)]
\end{align*}
Therefore, the number of occurrences of $e_0$ in $[f^j(\gamma)]$ is bounded by the number of occurrences of $e_0$ in $\gamma$.  The Lemma follows.
\end{proof}

\begin{lem}\label{red-neg-inverse}
Let $N \geq 2$.  Let $\varphi \in \Out(F_N)$ be so that at least one of $\varphi, \varphi^{-1}$ is an NEG-automorphism.  Then there is a good relative train track $f \colon G \to G$ and a topological representative $f' \colon G \to G$ representing $\varphi$ and $\varphi^{-1}$, not necessarily respectively, so that

\begin{enumerate}
\item $f'|_{G \backslash \{e_0\}} \colon G \backslash \{e_0\} \to G \backslash \{e_0\}$ is a homotopy inverse for $f|_{G_{t-1}} \colon G_{t-1} \to G_{t-1}$.
\item $f'(e_0) = e_0 u'$ for some $u' \Subset G \backslash \{e_0\}$.
\end{enumerate}

\end{lem}

\begin{proof}
Without loss of generality, assume that $\varphi$ is an NEG-automorphism equipped with $f \colon G \to G$.  By Theorem \ref{BFH5.1.5}, we may assume $E(H_t) = e_0$, a single edge and that $f(e_0) = e_0 u$ where $u \subseteq G_{t-1}$ is a closed path whose basepoint is fixed by $f$.  Note also that $o(e_0)$ is also fixed by $f$.  Let us first consider the case where $e_0$ is a $G$-separating edge.  Pass to the square of $f$ so that each component of $G_{t-1} = G_1 \sqcup G_2$ is fixed.  Without loss, say $o(e_0) \in V(G_1)$ and $t(e_0) \in V(G_2)$.  Let $\eta \colon \pi_1(G_1,o(e_0)) \to \pi_1(G_1,o(e_0))$ and $\zeta \colon \pi_1(G_2,t(e_0)) \to \pi_1(G_2,t(e_0))$ be the isomorphisms induced by $f$.  Let $u' \in \pi_1(G_2,t(e_0))$ be so that $\zeta(u') = \overline{u}$.  Let $f'_{\eta^{-1}}, f'_{\zeta^{-1}}$ be topological representatives for $\eta^{-1}$ and $\zeta^{-1}$ which fix the basepoints $o(e_0),t(e_0)$, respectively.  Put
\[
f' \colon e \mapsto \begin{cases}
f'_{\eta^{-1}}(e) &\text{if} \quad e \in \Gamma_{1} \\
f'_{\zeta^{-1}}(e) &\text{if} \quad e \in \Gamma_{2} \\
e u' &\text{if} \quad e = e_0 \end{cases}
\]
Then
\[
(f' \circ f) : e_0 \mapsto e_0 u' \mapsto e_0 u \overline{u}
\]
which reduces to $e_0$.  Evidently $f, f'$ are the desired maps.  The other case follows similarly.

\end{proof}

\begin{rem}\label{neg}
Using the topological representative from Lemma \ref{red-neg-inverse}, one can now obtain the conclusion of Lemma \ref{red-edge-neg} for $\varphi^{-1}$.
\end{rem}

\begin{thm}\label{red-neg-main}
Let $N \geq 2$.  Let $\varphi \in \Out(F_N)$ be so that at least one of $\varphi, \varphi^{-1}$ is an NEG-automorphism.  Then $\varphi$ has $\mathcal{P}$.  
\end{thm}

\begin{proof}
Without loss, say $\varphi$ is an NEG-automorphism, and let $a \in F_N$ be as per Lemma \ref{red-edge-neg}.  Then by Corollary \ref{marked-P} and Corollary \ref{ptp-one-sided-out} $\varphi$ has $\mathcal{P}^+(\mathcal{A},a)$ where $\mathcal{A}$ is a free basis of $F_N$ with $a \in \mathcal{A}$.  Corollary \ref{marked-P}, Corollary \ref{ptp-one-sided-out} and the topological representative constructed in Lemma \ref{red-neg-inverse} ensure that $\varphi$ has $\mathcal{P}^-(\mathcal{A},a)$ (see Remark \ref{neg}).  Thus $\varphi$ has $\mathcal{P} = \mathcal{P}(\mathcal{A},a)$.

\end{proof}

\section{The Case of $F_2$}\label{The Case of F2}

We now handle the case where $N=2$.  As mentioned in Section \ref{Preliminaries}, Cohen, Lustig, and Steiner \cite{MR1105334} showed that no finite subset of $F_2$ is spectrally rigid.  In fact, their argument works for any $\Sigma \subseteq F_2$ (possibly infinite) which satisfies a property, denoted $\mathcal{W}^*$, which we define below.

\begin{defn}\label{Ws}
Let $\Sigma \subseteq F_2$.  We say $\Sigma$ has property $\mathcal{W}^*$ if there is a free basis $F_2 = F(a,b)$ and $k \geq 1$ so that for all $\sigma \in \Sigma$, we have
\begin{enumerate}
\item $[[\sigma]]_{{\{a,ba^k\}}} \not \in \langle ba^k \rangle$.
\item If $(ba^k)^t \Subset [[\sigma]]_{{\{a,ba^k\}}}$, then $t = \pm1$. 
\end{enumerate}
\end{defn}

In view of the argument in \cite{MR1105334}, the following is immediate.

\begin{prop}\label{CLS}
Suppose $\Sigma \subseteq F_2$ has property $\mathcal{W}^*$.  Then $\Sigma$ is not spectrally rigid.
\end{prop} \qed

The following proposition allows us to extend our earlier arguments to the case of $F_2$.

\begin{prop}\label{WW}
Suppose $\Sigma \subseteq F_2$ has property $\mathcal{W}$.  Then $\Sigma$ has property $\mathcal{W}^*$.
\end{prop}

\begin{proof}
Let $\mathcal{A}, a \in \mathcal{A}$ and $M$ be as per property $\mathcal{W}$.  Write $\mathcal{A} = \{a,b\}$.  Choose $k > M$.  We first verify property 1 of Definition \ref{Ws}.  Suppose $[[\sigma]]_{\{ a , ba^k \}} = (ba^k)^t$ for $t \neq 0$.  Then $[[\sigma]]_{\mathcal{A}}$ contains $a^k$ (or $a^{-k}$) as a subword.  As $k > M$, this violates property $\mathcal{W}$.  Thus $[[\sigma]]_{{\{a,ba^k\}}} \not \in \langle ba^k \rangle$.  Note that if $[\sigma]_{c \mathcal{A}} = a^l$, then $[\sigma]_{c \{ a,ba^k \}} = a^l$, and so condition 2 is vacuously satisfied.  So write $[[\sigma]]_{\mathcal{A}} = a^{k_0^a} b^{k_0^b} \cdots a^{k_{n-1}^a} b^{k_{n-1}^b}$ where $k_{n-1}^b \neq 0$.  Then
\[
[[\sigma]]_{\{ a,ba^k \}} = a^{k_0^a} ((ba^k)(a^{-k}))^{k_0^b} \cdots a^{k_{n-1}^a} ((ba^k)(a^{-k}))^{k_{n-1}^b}
\]
Property $\mathcal{W}$ gives that $k_i^a \leq M$ for all $1 \leq i \leq n-1$.  Since $k > M$, we have that $|k_i^a - k| \neq 0$ for all $i$.  Thus we conclude that a nontrivial power of $a$ remains between successive occurrences of $ba^k$.  Thus we have property 2 of Definition \ref{Ws}.
\end{proof}

\section{Main Theorem}\label{Main Theorem}

We can now prove Theorem \ref{thma}, whose statement we recall for convenience.

\begin{thmaa}
Let $N \geq 2$.  Let $\Phi \in \Aut(F_N)$.  Let $g \in F_N$ be arbitrary.  Then the set $\Sigma = \{ \Phi^n(g) \}_{n \in \mathbb{Z}}$ is not spectrally rigid.
\end{thmaa}

\begin{proof}[Proof of Theorem \ref{thma}]
Suppose that $N \geq 2$ and let $\Phi \in \Aut(F_N)$ be arbitrary.  Let $[\Phi] = \varphi \in \Out(F_N)$.  If neither $\varphi$ nor $\varphi^{-1}$ is an NEG-automorphism, then $\varphi$ has $\mathcal{P}$ by Theorem \ref{red-eg-main}.  If at least one of $\varphi$ or $\varphi^{-1}$ is an NEG-automorphism, then $\varphi$ has $\mathcal{P}$ by Theorem \ref{red-neg-main}.  Thus $\varphi$ has $\mathcal{P}$.  Therefore, by Proposition \ref{PW}, for any $\Psi \in \varphi$ and for any $g \in F_N$ the set $\Sigma = \{ \Psi^n(g) \}_{n \in \mathbb{Z}}$ has $\mathcal{W}$.  Since $\Phi \in \varphi$, the set $\Sigma = \{ \Phi^n(g) \}_{n \in \mathbb{Z}}$ has $\mathcal{W}$.   If $N \geq 3$, then by Proposition \ref{SV}, $\Sigma$ is not spectrally rigid.  If $N=2$, then by Proposition \ref{WW}, $\Sigma$ has $\mathcal{W}^*$.  Thus by Proposition \ref{CLS}, $\Sigma$ is not spectrally rigid.
\end{proof}


\begin{thebibliography}{10}

\bibitem{MR1765705}
M.~Bestvina, M.~Feighn, and M.~Handel, \emph{The {T}its alternative for {${\rm
  Out}(F_n)$}. {I}. {D}ynamics of exponentially-growing automorphisms}, Ann. of
  Math. (2) \textbf{151} (2000), no.~2, 517--623.

\bibitem{MR1147956}
M.~Bestvina and M.~Handel, \emph{Train tracks and automorphisms of free
  groups}, Ann. of Math. (2) \textbf{135} (1992), no.~1, 1--51.

\bibitem{MR2396717}
O.~Bogopolski, \emph{Introduction to group theory}, EMS Textbooks in
  Mathematics, European Mathematical Society (EMS), Z\"urich, 2008, Translated,
  revised and expanded from the 2002 Russian original.

\bibitem{2011arXiv1106.0688F}
M.. {Carette}, S.~{Francaviglia}, I.~{Kapovich}, and A.~{Martino},
  \emph{{Spectral rigidity of automorphic orbits in free groups}}, preprint, June 2011; arXiv:1106.0688F.
  
\bibitem{MR1851337}
I.~Chiswell, \emph{Introduction to {$\Lambda$}-trees}, World Scientific
  Publishing Co. Inc., River Edge, NJ, 2001.

\bibitem{MR1105334}
M.~M. Cohen, M.~Lustig, and M.~Steiner, \emph{{${\bf R}$}-tree actions are not
  determined by the translation lengths of finitely many elements}, Arboreal
  group theory ({B}erkeley, {CA}, 1988), Math. Sci. Res. Inst. Publ., vol.~19,
  Springer, New York, 1991, pp.~183--187.

\bibitem{MR916179}
D.~Cooper, \emph{Automorphisms of free groups have finitely generated fixed
  point sets}, J. Algebra \textbf{111} (1987), no.~2, 453--456.

\bibitem{MR2365352}
R.~Geoghegan, \emph{Topological methods in group theory}, Graduate Texts in
  Mathematics, vol. 243, Springer, New York, 2008.

\bibitem{MR1086649}
{\'E}.~Ghys and P.~de~la Harpe, \emph{Panorama}, Sur les groupes hyperboliques
  d'apr\`es {M}ikhael {G}romov ({B}ern, 1988), Progr. Math., vol.~83,
  Birkh\"auser Boston, Boston, MA, 1990, pp.~1--25.

\bibitem{2010arXiv1001.1729K}
I.~{Kapovich}, \emph{{Random length-spectrum rigidity for free groups}}, Proceedings of AMS, to appear; arXiv:1001.1729K

\bibitem{MR1182503}
J.~Smillie and K.~Vogtmann, \emph{Length functions and outer space}, Michigan
  Math. J. \textbf{39} (1992), no.~3, 485--493.

\bibitem{MR1950871}
K.~Vogtmann, \emph{Automorphisms of free groups and outer space}, Proceedings
  of the {C}onference on {G}eometric and {C}ombinatorial {G}roup {T}heory,
  {P}art {I} ({H}aifa, 2000), vol.~94, 2002, pp.~1--31.

\end{thebibliography}
\end{document}